\documentclass[ctagsplt,12pt]{amsart}
\RequirePackage{newlfont}
\usepackage{amscd}
\usepackage{amsmath}
\usepackage{amsfonts}
\usepackage{amssymb}
\usepackage{enumerate}
\usepackage{color}

\begin{document}
\title{On Gregus-\'{C}iri\'{c} mappings on weighted graphs}
\author{M. R. Alfuraidan $\&$ M. A. Khamsi}

\address{Monther Rashed Alfuraidan\\Department of Mathematics \& Statistics \\
King Fahd University of Petroleum and Minerals\\
Dhahran 31261, Saudi Arabia.}
\email{monther@kfupm.edu.sa}
\address{Mohamed Amine Khamsi\\Department of Mathematical Sciences, University of Texas at El Paso, El Paso, TX 79968, USA.}
\email{mohamed@utep.edu}


\maketitle
\newtheorem{theorem}{Theorem}[section]
\newtheorem{acknowledgement}{Acknowledgement}
\newtheorem{algorithm}{Algorithm}
\newtheorem{axiom}{Axiom}
\newtheorem{case}{Case}
\newtheorem{claim}{Claim}
\newtheorem{conclusion}{Conclusion}
\newtheorem{condition}{Condition}
\newtheorem{conjecture}{Conjecture}
\newtheorem{corollary}{Corollary}[section]
\newtheorem{criterion}{Criterion}
\newtheorem{definition}{Definition}[section]
\newtheorem{example}{Example}[section]
\newtheorem{exercise}{Exercise}
\newtheorem{lemma}{Lemma}[section]
\newtheorem{notation}{Notation}
\newtheorem{problem}{Problem}
\newtheorem{proposition}{Proposition}[section]
\newtheorem{remark}{Remark}[section]
\newtheorem{solution}{Solution}
\newtheorem{summary}{Summary}
\newtheorem{property}{Property}

\thispagestyle{empty}

\begin{abstract}
\noindent \itshape In this paper, we introduce the concept of monotone Gregus-\'{C}iri\'{c}-contraction mappings in weighted digraphs.  Then we establish a fixed point theorem for monotone Gregus-\'{C}iri\'{c}-contraction mappings defined in convex weighted digraphs.
\end{abstract}

\section{Introduction}
Banach's Contraction Principle \cite{banach} is perhaps the most widely applied fixed point theorem in all of analysis. Over the years, many mathematicians tried successfully to extend this fundamental theorem.  In 1980, Gregus \cite{gregus} proved the following result:

\begin{theorem} Let $X$ be a Banach space and $C$ be a nonempty closed and convex subset of $X$. Let $T: C\rightarrow C$ be a mapping satisfying
$$\| T(x)-T(y)\| \le a\|x-y\|+p\|T(x)-x\|+p\|T(y)-y\|,$$
for all $x,y\in C$, where $0<a<1$, $p\ge 0$ and $a+2p=1$. Then $T$ has a unique fixed point.
\end{theorem}

\'{C}iri\'{c} \cite{ciric} obtained the following generalization of Gregus' theorem.

\begin{theorem} Let $(X,d)$ be a complete convex metric space and $C$ be a nonempty closed and convex subset of $X$. Let $T: C\rightarrow C$ be a mapping satisfying
\begin{equation}\label{CG}\tag{CG} \begin{array}{lll}
d(T(x),T(y))&\le& a\ \max\Big\{d(x,y),c\ \Big[d(x,T(y))+d(y,T(x))\Big]\Big\}\\
&&\\
&&\;\;\;\;\;\;\;\;\;\;+ b \ \max\{d(x,T(x)),d(y,T(y))\},
\end{array}\end{equation}
for all $x,y\in C$, where $0<a<1$, $a+b=1$ and $0\le c\le \frac{4-a}{8-a}$. Then $T$ has a unique fixed point.
\end{theorem}

\medskip
\begin{remark}  If we assume that $a+b < 1$ and $c \le \frac{1}{2}$, then any map $T$ which satisfies the condition (\ref{CG}) also satisfies the following condition:
$$d(T(x),T(y)) \leq (a+b)\ \max\Big\{ d(x,y), d(x,T(y)), d(y,T(x)), d(x,T(x)), d(y,T(y))\Big\}.$$
In other words, $T$ is a \'{C}iri\'{c} quasi-contraction mapping.  This concept was introduced by \'{C}iri\'{c} \cite{ciric-74} as an extension to the contraction condition.  In \cite{ciric-74}, he proved an analogue to the Banach Contraction Principle for this type of mappings without the use of convexity.
\end{remark}

Recently, Djafari-Rouhani and Moradi \cite{rouhanim} obtained the following improvement of \'{C}iri\'{c}'s result:
\begin{theorem}\label{rouhanimoradi} Let $(X,d)$ be a complete convex metric space and $T: X\rightarrow X$ be a mapping satisfying
\[\begin{array}{lll}
d(T(x),T(y))&\le& a\ \max\{d(x,y),c[d(x,T(y))+d(y,T(x))]\}\\
&& +b\ \max\{d(x,T(x)),d(y,T(y))\},
\end{array}\]
for all $x,y\in X$, where $0<a<1$, $a+b=1$ and $0\le c < \frac{1}{2}$. Then $T$ has a unique fixed point.
\end{theorem}
\medskip

\noindent In fact, the authors in \cite{rouhanim} gave a simple example which shows that the conclusion of Theorem \ref{rouhanimoradi} does not hold if $c > \frac{1}{2}$ and asked whether its conclusion holds when $c = \frac{1}{2}$.  This problem is still open.\\
\medskip

In this work, we generalize Theorem \ref{rouhanimoradi} to the case of monotone self-mappings defined on a weighted graph.

\section{Preliminaries}

\noindent A graph $G$ is a nonempty set $V(G)$ of elements called vertices together with a possibly empty subset $E(G)$ of $V(G)\times V(G)$ called edges. A directed graph (digraph) is a graph with a direction assigned to each of its edges. In this paper, we assume that all digraphs are reflexive, i.e., $(x, x) \in E(G)$ for each $x\in V(G) $. Moreover, we assume that there exists a distance function $d$ defined on the set of vertices $V(G)$. Throughout this work, we treat $G$ as a weighted digraph by giving each edge the metric distance between its vertices.

Let $x$ and $y$ be in $V(G)$.  A (directed) path from $x$ to $y$ is a finite sequence $\{x_{i}\}_{i=0}^{N}$ of vertices such that $x_{0} = x$, $x_{N} = y$ and $(x_{i-1},x_{i})\in E(G)$ for $i = 1,...,N$.

\begin{definition}
The digraph $G$ is said to be transitive if $(x,z) \in E(G)$ whenever  $(x,y)\in E(G)$ and $(y,z)\in E(G)$, for any $x,y,z\in V(G)$. In another words, $G$ is transitive if for any two vertices $x$ and $y$ that are connected by a directed finite path, we have $(x,y) \in E(G)$.
\end{definition}

\begin{definition} \label{monotone-GC-mapping}  Let $G$ be a weighted digraph and $d$ be a metric distance on $V(G)$. Let $C$ be a nonempty subset of $V(G)$.  A mapping $T: C \rightarrow C$ is called
\begin{enumerate}
\item[(1)]  $G$-monotone if $T$ is edge preserving, i.e., $(T(x),T(y))\in E(G)$ whenever $(x,y)\in E(G)$, for any $x, y \in C$.
\item[(2)]  $G$-monotone Gregus-\'{C}iri\'{c}-mapping if $T$ is $G$-monotone and there exist $a, b, c \in [0,+\infty)$ such that
\[\begin{array}{lll}
d(T(x),T(y))&\le& a\ \max \Big\{d(x,y),c\Big[d(x,T(y))+d(y,T(x))\Big]\Big\}\\
&&+b\ \max \Big\{d(x,T(x)),d(y,T(y))\Big\},
\end{array}\]
for any $x,y\in C$ with $(x,y)\in E(G)$.
\item[(3)] $G$-monotone Gregus-\'{C}iri\'{c}-contraction if $T$ is $G$-monotone Gregus-\'{C}iri\'{c}-mapping for which $0<a<1$, $a+b = 1$ and $c \leq \frac{1}{2}$.
\end{enumerate}
The point $x \in C$ is called a fixed point of $T$ if $T(x) = x$.
\end{definition}

\medskip
Note that in the example given by the authors in \cite{rouhanim}, the mapping $T(x) = x+1$ is monotone for the order and may be seen as an example of a monotone Gregus-\'{C}iri\'{c} mapping.  Moreover, any monotone contraction is a monotone Gregus-\'{C}iri\'{c}-contraction. The example studied by Ran and Reurings \cite{RR} gives an example of a monotone-contraction which fails to be a contraction.

\medskip\medskip
The following definition is needed since we will be using the concept of increasing or decreasing sequences in the sense of a digraph.

\begin{definition}  Let $G$ be a digraph.  A sequence $\{x_n\} \in V(G)$ is said to be
\begin{itemize}
\item[(a)] $G$-increasing if $(x_n,x_{n+1}) \in E(G)$, for all $n \in \mathbb{N}$;
\item[(b)] $G$-decreasing if $(x_{n+1},x_{n}) \in E(G)$, for all $n \in \mathbb{N}$;
\item[(c)] $G$-monotone if $\{x_n\}$ is either $G$-increasing or $G$-decreasing.
\end{itemize}
\end{definition}

\begin{definition}\label{G-complete}  Let $G$ be a weighted digraph and $d$ be a metric distance on $V(G)$.  A subset $C$ of $V(G)$ is said to be $G$-complete if any $G$-monotone sequence $\{x_n\}$ in $C$ which is Cauchy is convergent to a point in $C$.
\end{definition}

\medskip
\begin{remark}\label{G-complete}  Let $G$ be a weighted digraph and $d$ be a metric distance on $V(G)$.  If $(V(G), d)$ is complete, then $V(G)$ is $G$-complete.  The converse is not true.  Indeed, let $X = \{(x,y) \in \mathbb{R}^2;\ 0 \leq x < 1\; and\;\; 0 \le y \le 1\}$ endowed with the Euclidean distance.  Then $(X,d)$ is not complete.  Moreover, if we consider the weighted digraph $G$ such that $V(G) = X$ and
$$\Big((x,y), (a,b)\Big) \in V(G) \;\; \mbox{if and only if}\;\; x = a \; and\;\; y \le b;$$
then $V(G)$ is $G$-complete.  In fact, any $G$-monotone sequence is convergent.
\end{remark}

\medskip
\noindent As Jachymski did in \cite{Jachymski}, we introduce the following property:

\begin{definition}\label{Jachymski-*}
Let $G$ be a weighted digraph and $C$ be a nonempty subset of $V(G)$. We say that $C$ has Property (*) if for any $G$-increasing (resp. $G$-decreasing) sequence $\{x_{n}\}$ in $C$ which converges to  $x$,  there is a subsequence $\{x_{k_{n}}\}$ with  $(x_{k_{n}}, x)\in E(G)$ (resp. $(x, x_{k_{n}})\in E(G)$), for $n \in \mathbb{N}$.
\end{definition}
\medskip

\noindent Note that if $G$ is transitive, then Property (*) implies that for any $G$-increasing sequence $\{x_n\}$ (resp. $G$-decreasing) which converges to $x$, we have $(x_{n}, x)\in E(G)$ (resp. $(x, x_n)\in E(G)$), for every $n \in \mathbb{N}$.

\section{Some basic results}
Throughout this section, we consider $G$ a weighted digraph with $d$ a metric distance on $V(G)$. Let $C$ be a nonempty subset of $V(G)$ and $T: C \rightarrow C$ be $G$-monotone Gregus-\'{C}iri\'{c}-contraction mapping.  Then there exist positive numbers $a, b, c $ such that $0 < a <1$, $a+b = 1$ and $c \le \frac{1}{2}$ such that
$$\begin{array}{lll}
d(T(x),T(y))&\le& a\ \max \Big\{d(x,y),c\Big[d(x,T(y))+d(y,T(x))\Big]\Big\}\\
&&\;\;\;\;\; +\ b\ \max \Big\{d(x,T(x)),d(y,T(y))\Big\},
\end{array}$$
for any $x,y\in C$ with $(x,y)\in E(G)$.\\

The following technical results will be crucial to the establishment of the main theorem of this work.

\begin{lemma}\label{lemma1}  Under the above assumptions, we have
$$d(x,y) \leq \frac{2-a}{1-a} \ \Big(d(x,T(x)) + d(y, T(y))\Big),$$
for any $x,y\in C$ with $(x,y)\in E(G)$  or $(y,x)\in E(G)$.
\end{lemma}
\begin{proof} Without loss of generality, we assume $(x,y)\in E(G)$.  Then we have
\[\begin{array}{lll}
d(T(x),T(y))&\le& a\ \max \Big\{d(x,y),c\Big[d(x,T(y))+d(y,T(x))\Big]\Big\}\\
&&\;\;\; +b\ \max \Big\{d(x,T(x)),d(y,T(y))\Big\}.
\end{array}\]
Since $c \leq \frac{1}{2}$, we get
\[\begin{array}{lll}
c\Big[d(x,T(y))+d(y,T(x))\Big] &\leq& c\Big[d(x,T(x)) + 2 d(T(x), T(y))+ d(y,T(y))\Big]\\
&\le & d(x,T(x)) + d(T(x), T(y))+ d(y,T(y),
\end{array}\]
which implies
\[\begin{array}{lll}
d(T(x),T(y))&\le& a\ \max \Big\{d(x,T(x)) + d(T(x), T(y)) + d(y, T(y)),\\
&&\;\;\; c\Big[d(x,T(y))+d(y,T(x))\Big]\Big\} +b\ \max \Big\{d(x,T(x)),d(y,T(y))\Big\}\\
&\le & a\ \Big\{d(x,T(x)) + d(T(x), T(y)) + d(y, T(y))\Big\} \\
&& \;\;\;  + b\ \Big\{d(x,T(x))+ d(y,T(y))\Big\}\\q

&\le & (a+b) \Big(d(x,T(x)) + d(T(x), T(y))\Big) + a \ d(T(x), T(y)).
\end{array}\]
Since $a+b = 1$, we get
$$d(T(x),T(y))\le \frac{1}{1-a}\ \Big(d(x,T(x)) + d(T(x), T(y))\Big).$$
Hence
\[\begin{array}{lll}
d(x,y)&\le& d(x,T(x)) + d(T(x), T(y)) + d(y, T(y),\\
&\le &\left(1+\frac{1}{1-a}\right)\ \Big(d(x,T(x)) + d(T(x), T(y))\Big),
\end{array}\]
which implies
$$d(x,y) \leq \frac{2-a}{1-a} \ \Big(d(x,T(x)) + d(y, T(y))\Big).$$
\end{proof}

\medskip

\begin{lemma}\label{lemma2}  Under the above assumptions, if $x \in C$ is such that $(x, T(x)) \in E(G)$ or $(T(x),x) \in E(G)$, then the sequence $\Big \{ d(T^n(x), T^{n+1}(x))\Big \}_{n \in \mathbb{N}}$ is decreasing.
\end{lemma}
\begin{proof} Without loss of generality, we assume $(x,T(x))\in E(G)$.  Since $T$ is $G$-monotone, we get $(T^n(x), T^{n+1}(x)) \in E(G)$, for any $n \in \mathbb{N}$.  Fix $n \ge 1$.  Then
\[\begin{array}{lll}
d(T^n(x),T^{n+1}(x))&\le& a\ \max \Big\{d(T^{n-1}(x),T^n(x)),c\ d(T^{n-1}(x),T^{n+1}(x))\Big\}\\
&&\;\;\; +\ b\ \max \Big\{d(T^{n-1}(x),T^n(x)),d(T^n(x),T^{n+1}(x))\Big\}.\\
\end{array}\]
Assume that $d(T^{n-1}(x),T^n(x)) < d(T^n(x),T^{n+1}(x))$ holds.  Since
$$\begin{array}{lll}
d(T^{n-1}(x),T^{n+1}(x)) &\leq& d(T^{n-1}(x),T^{n}(x)) + d(T^{n}(x),T^{n+1}(x))\\
& < & 2\ d(T^{n}(x),T^{n+1}(x)),
\end{array}$$
and $c \le \frac{1}{2}$, we get
$$\begin{array}{lll}
d(T^{n}(x),T^{n+1}(x)) & < & a\ d(T^{n}(x),T^{n+1}(x)) + b\ d(T^{n}(x),T^{n+1}(x))\\
& = & d(T^{n}(x),T^{n+1}(x)).
\end{array}$$
This contradiction forces $d(T^n(x),T^{n+1}(x)) \leq d(T^{n-1}(x),T^n(x))$.  Since $n$ was taken arbitrarily, we conclude that $\Big \{d(T^n(x), T^{n+1}(x))\Big \}_{n\in \mathbb{N}}$ is decreasing.
\end{proof}

\begin{lemma}\label{lemma3}  Under the above assumptions, assuming $G$ is transitive, if $x \in C$ such that $(x, T(x)) \in E(G)$ or $(T(x),x) \in E(G)$, then there exists $n \geq 1$ such that
$$d(T^n(x),T^{n+2}(x)) \leq \frac{2}{2-a}\ d(x,T(x)).$$
\end{lemma}
\begin{proof} Here we mimic an argument used by Djafari-Rouhani and Moradi in their proof of [\cite{rouhanim}, Theorem 2.2]. Without loss of generality, we assume that $(x, T(x)) \in E(G)$.  Since $G$ is transitive and $T$ is $G$-monotone, then $(T^n(x), T^{n+h}(x)) \in E(G)$, for any $n, h \in \mathbb{N}$.  Fix $n \ge 1$.  then we have
\[\begin{array}{lll}
d(T^n(x),T^{n+2}(x))&\le& a\ \max \Big\{d(T^{n-1}(x),T^{n+1}(x)),\\
&&\;\;\;   c\ \Big[d(T^{n-1}(x),T^{n+2}(x))+\ d(T^{n}(x),T^{n+1}(x))\Big]\Big\} \\
&&\;\;\;+\ b\ \max \Big\{d(T^{n-1}(x),T^n(x)),d(T^{n+1}(x),T^{n+2}(x))\Big\}.\\
\end{array}\]
Assume that for some $n \geq 1$, we have
$$d(T^{n-1}(x),T^{n+1}(x)) \le c\ \Big[d(T^{n-1}(x),T^{n+2}(x))+\ d(T^{n}(x),T^{n+1}(x))\Big].$$
Since $\Big\{d(T^n(x), T^{n+1}(x))\Big\}_{n\in \mathbb{N}}$ is decreasing, we get
\[\begin{array}{lll}
d(T^n(x),T^{n+2}(x))&\le& ac\ \Big[d(T^{n-1}(x),T^{n+2}(x))+\ d(T^{n}(x),T^{n+1}(x))\Big] \\
&&\;\;\;+\ b\ d(x,T(x)).\\
\end{array}\]
Since $d(T^{n-1}(x),T^{n+2}(x)) \le d(T^{n-1}(x),T^{n}(x)) + d(T^{n}(x),T^{n+2}(x))$, we get
\[\begin{array}{lll}
d(T^n(x),T^{n+2}(x))&\le& ac\ \Big[2\ d(x,T(x))+\ d(T^{n}(x),T^{n+2}(x))\Big] \\
&&\;\;\;+\ b\ d(x,T(x)),\\
\end{array}\]
which implies
$$d(T^n(x),T^{n+2}(x)) \leq \frac{2ac + b}{1-ac}\ d(x, T(x)).$$
The function $f(c) = \frac{2ac + b}{1-ac}$ is increasing in the interval $[0, \frac{1}{2}]$.  Hence
$$\frac{2ac + b}{1-ac} \leq \frac{a + b}{1-a/2} = \frac{2}{2-a}.$$
Therefore, we have
$$d(T^n(x),T^{n+2}(x)) \leq \frac{2}{2-a}\ d(x, T(x)).$$
Next, assume that for any $n \ge 1$, we have
$$c\ \Big[d(T^{n-1}(x),T^{n+2}(x))+\ d(T^{n}(x),T^{n+1}(x))\Big] \le d(T^{n-1}(x),T^{n+1}(x)).$$
In this case, we have
$$d(T^n(x),T^{n+2}(x))\le a\ d(T^{n-1}(x), T^{n+1}(x)) + b \ d(x,T(x)),$$
which easily implies
$$\begin{array}{lll}
d(T^n(x),T^{n+2}(x)) &\le& a^{n-1}\ d(x, T^2(x)) + \frac{b}{1-a}\ d(x,T(x)\\
&=& a^{n-1}\ d(x, T^2(x)) + d(x,T(x).
\end{array}$$
Since $d(x, T^2(x)) \le d(x, T(x)) + d(T(x), T^2(x)) \le 2\ d(x,T(x))$, we conclude that
$$d(T^n(x),T^{n+2}(x))\le(2\ a^{n-1}\ + 1)\ d(x,T(x),$$
for any $n \geq 1$.  Since $0 < a < 1$, there exists $n \geq 1$ such that
$$2\ a^{n-1}\ + 1 \leq \frac{a}{2-a} + 1 = \frac{2}{2-a},$$
which implies
$$d(T^n(x),T^{n+2}(x)) \leq \frac{2}{2-a}\ d(x, T(x)).$$
\end{proof}

\medskip
Our final basic result of this section is the following:

\begin{lemma}\label{lemma4}  Let $a, b, c $ be positive numbers such that $0 < a <1$, $a+b = 1$ and $c < \frac{1}{2}$.  Then if we choose  $\beta \ge 0$ such that $2c < \beta <1$, we have
$$K = \alpha \ a \ \max\Big\{ \alpha + \frac{2\beta}{2-a}, c\ \Big[2 + \frac{2\beta}{2-a}\Big]\Big\} + \beta^2\ a + b < 1,$$
where $\alpha = 1 - \beta$.
\end{lemma}
\begin{proof} Note that $K < 1$ if and only if
$$\alpha \ \max\Big\{ \alpha + \frac{2\beta}{2-a}, c\ \Big[2 + \frac{2\beta}{2-a}\Big]\Big\} + \beta^2\ < 1,$$
where we used $1-b = a$ and $a > 0$.  Since $1 - \beta^2 = \alpha (1 + \beta)$ and $\alpha > 0$, we get $K < 1$ if and only if
$$\max\Big\{ \alpha + \frac{2\beta}{2-a}, c\ \Big[2 + \frac{2\beta}{2-a}\Big]\Big\}  < 1 + \beta.$$
Since $ a < 1$, we get $1 < 2-a$ which implies $\displaystyle \frac{2\beta}{2-a} < 2 \beta$.  Hence
$$\alpha + \frac{2\beta}{2-a} = 1 - \beta + \frac{2\beta}{2-a} < 1 + \beta.$$
Moreover, we have $\beta^2 < 1 < 2-a$ and since $2c < \beta$, we get
$$2c\ \Big[ 1 + \frac{\beta}{2-a}\Big] < \beta \ \Big[ 1 + \frac{\beta}{2-a}\Big] = \beta + \frac{\beta^2}{2-a} < 1 + \beta. $$
Therefore, we have
$$\max\Big\{ \alpha + \frac{2\beta}{2-a}, c\ \Big[2 + \frac{2\beta}{2-a}\Big]\Big\}  < 1 + \beta,$$
which completes the proof that $K < 1$.
\end{proof}

\medskip
In the next section, we discuss the existence of fixed points of $G$-monotone Gregus-\'{C}iri\'{c} mappings defined in weighted graphs.

\section{Fixed Points of $G$-Monotone Gregus-\'{C}iri\'{c}-nonexpansive Mappings}
As we said earlier, the fixed point results obtained for these type of mappings were done in the context of convex metric spaces.  Convexity in metric spaces was initiated by Menger \cite{menger} in 1928.  The terms "metrically convex" and convex metric space" are due to Blumenthal \cite{blumenthal}.   

\medskip

\noindent Let $(M, d)$ be a metric space and let $ \mathbb{R }$ denotes the real line. We
say that a mapping $c : \mathbb{R} \rightarrow M$ is a metric embedding of $ \mathbb{R }$ into
$M$ if $$d(c(s),c(t))=|s-t|,$$ for all real $s,t\in \mathbb{R }$.
\begin{itemize}
  \item[(i)] The image $c([a, b]) \subset M$ of a real interval under a metric
embedding will be called a metric segment, also known as a geodesic in the literature.
  \item[(ii)] Let $x, y \in M$. A metric segment $c([a, b])$ is said to join $x$ and $y$ if $c(a) = x$ and $c(b) = y$ and will be denoted by $[x,y]$.
  \item[(iii)]  $C\subset M $  is said to be convex whenever $[x,y] \subset C$ for any $x, y \in C$.
\end{itemize}
Assume that for any $x$ and $y$ in $M$, there exists a unique metric segment $[x,y]$.  For any $\beta \in [0,1]$, the unique point $z \in [x,y]$ such that
\[d(x,z) = (1-\beta) d(x,y), \;\; \mbox{and}\;\;\; d(z,y) = \beta d(x,y),\]
will be denoted by $\beta x \oplus (1-\beta) y$. Metric spaces having this property are usually called convex metric spaces or geodesic metric spaces \cite{menger, takahashi}.  Moreover, in this section, we assume that
$$d(\beta x \oplus (1-\beta) y, z) \leq \beta d(x,z) + (1-\beta) d(y,z),$$
for any $x, y, z \in M$ and $\beta \in [0,1]$.  This property of the metric convex combination was introduced by Takahashi in \cite{takahashi}.  Normed vector spaces and hyperbolic metric spaces \cite{r-s} are a natural example of convex metric spaces which satisfy all the above properties.\\

Throughout this section, we consider $G$ a transitive weighted digraph with $d$ a metric distance on $V(G)$. We assume that $V(G)$ is a convex metric space such that $G$-intervals are convex.  Recall that a $G$-interval is any of the subsets
$$[x, \rightarrow) = \{y \in V(G);\ (x,y) \in E(G)\}\; or \; (\leftarrow, x] = \{y \in V(G);\ (y,x) \in E(G)\}.$$
Now, we are ready to state the main fixed point result of this work.

\begin{theorem}\label{main}  Let $C$ be a nonempty $G$-complete and convex subset of $V(G)$ which satisfies the Property (*).  Let $T: C \rightarrow C$ be $G$-monotone Gregus-\'{C}iri\'{c}-contraction mapping, i.e. there exist positive numbers $a, b, c $ such that $0 < a <1$, $a+b = 1$ and $c \le \frac{1}{2}$ such that
$$\begin{array}{lll}
d(T(x),T(y))&\le& a\ \max \Big\{d(x,y),c\Big[d(x,T(y))+d(y,T(x))\Big]\Big\}\\
&&\;\;\;\;\; +\ b\ \max \Big\{d(x,T(x)),d(y,T(y))\Big\},
\end{array}$$
for any $x,y\in C$ with $(x,y)\in E(G)$.  Assume that $c < \frac{1}{2}$.  Let $x \in C$ be such that $(x, T(x)) \in E(G)$ (or $(T(x),x) \in E(G)$).  Then $T$ has a fixed point $\omega$ such that $(x, \omega) \in E(G)$ (or $(\omega, x) \in E(G)$).  Moreover, if $\Omega$ is another fixed point of $T$ such that $(x, \Omega) \in E(G)$ (or $(\Omega, x) \in E(G)$), then we must have $\omega = \Omega$.
\end{theorem}
\begin{proof} Without loss of generality, we assume that $(x, T(x)) \in E(G)$ and $x$ is not a fixed point of $T$.  In this case, we have $(T^n(x), T^{n+1}(x)) \in E(G)$, for any $n \in \mathbb{N}$.  Lemma \ref{lemma3} implies the existence of $n \geq 1$ such that
$$d(T^n(x),T^{n+2}(x)) \leq \frac{2}{2-a}\ d(x,T(x)).$$
Let $\beta < 1$ be the number obtained in Lemma \ref{lemma4}.  Set
$$z = \alpha\ T^{n+1}(x) \oplus \beta\ T^{n+2}(x) \in C$$
since $C$ is convex.  Using the convexity of the $G$-intervals, we have $(T^{n+1}(x), z) \in E(G)$ and $(z, T^{n+2}(x)) \in E(G)$.  Since $T$ is $G$-monotone and $G$ is transitive, we conclude that $(z, T(z)) \in E(G)$ and $(T^n(x), z) \in E(G)$.  Moreover, we have
$$d(z,T(z)) \leq \alpha\ d(T^{n+1}(x), T(z)) + \beta\ d(T^{n+2}(x),T(z)).$$
Hence
$$\begin{array}{lll}
d(T^{n+1}(x), T(z)) &\leq& a \ \max \Big\{d(T^n(x), z), c\ \Big[d(T^{n+1}(x), z) + d(T^n(x), T(z))\Big]\Big\}\\
&&\;\;\;\;\;\;\;\; +\ b \ \max\Big\{d(T^n(x), T^{n+1}(x)), d(z,T(z))\Big\},
\end{array}$$
and
$$\begin{array}{lll}
d(T^{n+2}(x), T(z)) &\leq& a \ \max \Big\{d(T^{n+1}(x), z), c\ \Big[d(T^{n+2}(x), z) + d(T^{n+1}(x), T(z))\Big]\Big\}\\
&&\;\;\;\;\;\;\;\; +\ b \ \max\Big\{d(T^{n+1}(x), T^{n+2}(x)), d(z,T(z))\Big\}.
\end{array}$$
First note that we have
$$\begin{array}{lll}
d(T^n(x), z) &\le& \alpha \ d(T^n(x), T^{n+1}(x)) + \beta\ d(T^n(x), T^{n+2}(x))\\
&\le & \alpha \ d(x, T(x)) + \beta\ \frac{2}{2-a} \ d(x, T(x)),
\end{array}$$
and
$$\begin{array}{lll}
d(T^{n+1}(x), z) + d(T^n(x), T(z)) &\le& \beta \ d(T^{n+1}(x), T^{n+2}(x)) + d(T^n(x), z) + d(z,T(z))\\
&\le & \beta \ d(x, T(x)) + \alpha\ d(T^n(x), T^{n+1}(x))\\
 &&\;\;\;\;\;\; +\ \beta\ d(T^n(x), T^{n+2}(x)) + d(z,T(z)),\\
&\le & d(x, T(x)) + \beta\ \frac{2}{2-a}\ d(x, T(x)) + d(z, T(z)),
\end{array}$$
which implies
$$\begin{array}{lll}
d(T^{n+1}(x), T(z)) &\leq& a \ \max \Big\{\Big[\alpha + \frac{2 \beta}{2-a}\Big]\ d(x, T(x)), c\ \Big[\Big(1 + \frac{2\beta}{2-a}\Big)\ d(x, T(x)) \\
&&\;\;\; + d(z, T(z))\Big]\Big\} +\ b \ \max\Big\{d(T^n(x), T^{n+1}(x)), d(z,T(z))\Big\}\\
&\le & a \ \max \Big\{\Big[\alpha + \frac{2 \beta}{2-a}\Big]\ d(x, T(x)), c\ \Big[\Big(1 + \frac{2\beta}{2-a}\Big)\ d(x, T(x)) \\
&&\;\;\; + d(z, T(z))\Big]\Big\} +\ b \ \max\Big\{d(x, T(x)), d(z,T(z))\Big\}.
\end{array}$$
Similarly, we have
$$\begin{array}{lll}
d(T^{n+2}(x), T(z)) &\leq& a \ \max \Big\{d(T^{n+1}(x), z), c\ \Big[d(T^{n+2}(x), z) + d(T^{n+1}(x), T(z))\Big]\Big\}\\
&&\;\;\;\;\;\;\;\; +\ b \ \max\Big\{d(T^{n+1}(x), T^{n+2}(x)), d(z,T(z))\Big\}\\
&\le& a \ \max \Big\{\beta\ d(T^{n+1}(x), T^{n+2}(x)), c\ \Big[\alpha\ d(T^{n+2}(x), T^{n+1}(x)) \\
&&  +\ d(T^{n+1}(x),z) + d(z, T(z))\Big]\Big\} + b \ \max\Big\{d(x, T(x)), d(z,T(z))\Big\}\\
&\le& a \ \max \Big\{\beta\ d(x, T(x)), c\ \Big[\alpha\ d(x, T(x)) +\beta\ d(T^{n+1}(x),T^{n+2}(x)) \\
&&  + d(z, T(z))\Big]\Big\} + b \ \max\Big\{d(x, T(x)), d(z,T(z))\Big\}\\
&\le& a \ \max \Big\{\beta\ d(x, T(x)), c\ \Big[ d(x,T(x))+ d(z, T(z))\Big]\Big\} \\
&&\;\;\;\;\;\;\;\;   + b \ \max\Big\{d(x, T(x)), d(z,T(z))\Big\}.\\
\end{array}$$
Since
$$d(z,T(z)) \leq \alpha\ d(T^{n+1}(x), T(z)) + \beta\ d(T^{n+2}(x),T(z)),$$
we get
$$\begin{array}{lll}
d(z, T(z)) &\leq& \alpha\ a \ \max \Big\{\Big[\alpha + \frac{2 \beta}{2-a}\Big]\ d(x, T(x)), c\ \Big[\Big(1 + \frac{2\beta}{2-a}\Big)\ d(x, T(x)) + d(z, T(z))\Big]\Big\}\\
&&\;\;\; +\ \beta\ a \ \max \Big\{\beta\ d(x, T(x)), c\ \Big[ d(x,T(x))+ d(z, T(z))\Big]\Big\}\\
&&\;\;\; +\ b \ \max\Big\{d(x, T(x)), d(z,T(z))\Big\}.
\end{array}$$
Assume that $d(x,T(x)) < d(z,T(z))$.  Then, we must have
$$\begin{array}{lll}
d(z, T(z)) &<& \alpha\ a \ \max \Big\{\Big[\alpha + \frac{2 \beta}{2-a}\Big], c\ \Big[\Big(1 + \frac{2\beta}{2-a}\Big) + 1\Big]\Big\} d(z,T(z)\\
&&\;\;\; +\ \beta\ a \ \max \{\beta, 2\ c\}\ d(z,T(z)) +\ b \ d(z,T(z)).
\end{array}$$
Since $2c < \beta$, we get
$$d(z, T(z)) < \Big[\alpha\ a \ \max \Big\{\alpha + \frac{2 \beta}{2-a}, c\ \Big(2 + \frac{2\beta}{2-a}\Big)\Big\} +\ \beta^2\ a +\ b \Big]\ d(z,T(z)).$$
Using Lemma \ref{lemma4}, we know that
$$K = \alpha\ a \ \max \Big\{\alpha + \frac{2 \beta}{2-a}, c\ \Big(2 + \frac{2\beta}{2-a}\Big)\Big\} +\ \beta^2\ a +\ b < 1,$$
which implies $d(z,T(z)) < K \ d(z,T(z))$ an obvious contradiction.   Therefore, we must have $d(z,T(z)) \leq d(x,T(x))$.  Hence
$$d(z, T(z)) \le \Big[\alpha\ a \ \max \Big\{\alpha + \frac{2 \beta}{2-a}, c\ \Big(2 + \frac{2\beta}{2-a}\Big)\Big\} +\ \beta^2\ a +\ b \Big]\ d(x,T(x)),$$
i.e. $d(T(z),z) \le K\ d(x,T(x))$.  By induction, we will construct a sequence $\{z_n\}$ in $C$ such that
\begin{enumerate}
\item[(i)] $z_0 = x$ and $z_1$ is the point constructed before;
\item[(ii)] $(z_n,z_{n+1}) \in E(G)$, for any $n \in \mathbb{N}$;
\item[(iii)] $d(z_{n+1},T(z_{n+1})) \leq K\ d(z_n,T(z_n))$, for any $n \in \mathbb{N}$.
\end{enumerate}
In particular, we have $d(z_{n+1},T(z_{n+1})) \leq K^n\ d(x,T(x))$, for any $n \in \mathbb{N}$.  Since $G$ is transitive, then $(z_n, z_m) \in E(G)$ for any $n \le m$.  Using Lemma \ref{lemma1}, we get
$$d(z_n, z_m) \leq \frac{2-a}{1-a}\ \Big(d(z_n, T(z_n)) + d(z_m, T(z_m))\Big).$$
Since $K < 1$, we conclude that $\{z_n\}$ is Cauchy and $G$-increasing.  Hence it is convergent some point $\omega \in C$ because $C$ is $G$-complete.  Since $C$ satisfies the Property (*), we conclude that $(z_n,\omega) \in E(G)$, for any $n \in \mathbb{N}$.  In particular, we have $(x,\omega) \in E(G)$.  Next, we prove that $\omega$ is a fixed point of $T$.  Since $(z_n,\omega) \in E(G)$, for any $n \in \mathbb{N}$, we get
$$\begin{array}{lll}
d(T(z_n), T(\omega)) &\le & a \max\Big\{d(z_n, \omega), c\ \Big[d(z_n, T(\omega)) + d(T(z_n), \omega)\Big]\Big\}\\
&&\;\;\;\; +\ b \ \max\Big\{d(z_n, T(z_n)), d(\omega, T(\omega))\Big\}.
\end{array}$$
Since $\lim\limits_{n \rightarrow +\infty} d(z_n, T(z_n)) = \lim\limits_{n \rightarrow +\infty} d(z_n, \omega) = 0$, we get $\lim\limits_{n \rightarrow +\infty} d(T(z_n), \omega) = 0$, which implies
$$d(\omega, T(\omega)) \le  a \max\Big\{0, c\ \Big[d(\omega, T(\omega)) + 0\Big]\Big\} +\ b \ \max\Big\{0, d(\omega, T(\omega))\Big\},$$
i.e. $d(\omega, T(\omega)) \le  a\ c\ d(\omega, T(\omega))  +\ b \ d(\omega, T(\omega))$.  Since $ac + b < a+b  = 1$, we conclude that $d(\omega, T(\omega)) = 0$, i.e. $T(\omega) = \omega$.  Finally, let $\Omega$ be another fixed point of $T$ such that $(x, \Omega) \in E(G)$.  Since $T$ is $G$-monotone, we get $(T^n(x), \Omega) \in E(G)$.  Using the convexity of the $G$-intervals, we get $(z_n, \Omega) \in E(G)$ for any $n \in \mathbb{N}$.  Using Lemma \ref{lemma1}, we get
$$d(z_n, \Omega) \leq \frac{2-a}{1-a} \ \Big(d(z_n,T(z_n)) + d(\Omega, T(\Omega))\Big) = \frac{2-a}{1-a} \ d(z_n,T(z_n)),$$
for any $n \in \mathbb{N}$. If we let $n \rightarrow +\infty$, we conclude that $\{z_n\}$ converges to $\Omega$.  the uniqueness of the limit implies that $\omega = \Omega$.
\end{proof}

\medskip
\begin{remark}  If we assume that $a+b < 1$ and $c \le \frac{1}{2}$, then the map $T$ is a quasi-contraction mapping \cite{ciric-74}.  In this case, Theorem \ref{main} is similar to the main fixed point result found in \cite{monther1, bachar-amine} without any convexity assumption on the weighted graph.
\end{remark}

\medskip\medskip\medskip

\noindent {\large{\textbf{Acknowledgements}}}
\medskip \\
\noindent The authors would like to acknowledge the support provided by the Deanship of Scientific Research at King Fahd University of Petroleum \& Minerals for funding this work through project No. IN141040.

\medskip\medskip


\end{document}